\documentclass[10pt,letter]{amsart}
\usepackage[utf8]{inputenc}
\usepackage{amsmath, calligra, mathrsfs} 
\usepackage{enumerate, color, hyperref}
\usepackage{bbm}
\usepackage[all]{xy}
\usepackage{amsfonts} 
\usepackage{amssymb} 
\usepackage{graphics} 
\usepackage{amscd} 
\usepackage{enumitem} 
\usepackage{todonotes} 

\usepackage[OT2,T1]{fontenc}

\DeclareSymbolFont{cyrletters}{OT2}{wncyr}{m}{n}
\DeclareMathSymbol{\Sha}{\mathalpha}{cyrletters}{"58}
\newcommand{\mint}{\times\!\!\!\!\!\!\!\hspace{.097em}\int}
\newcommand{\modsymb}[2]{\left[\frac{#1}{#2}\right]} 

\makeatletter 
\newcommand\figcaption{\def\@captype{figure}\caption} 
\newcommand\tabcaption{\def\@captype{table}\caption} 
\makeatother

\usepackage[left=1in, right=1in, top=1in, bottom=1in]{geometry}

\newtheorem{thm}{Theorem}[section]
\newtheorem{prop}[thm]{Proposition}

\newtheorem{lem}[thm]{Lemma}
\newtheorem{con}[thm]{Conjecture}

\theoremstyle{definition}
\newtheorem{rmk}[thm]{Remark}
\newtheorem{assumption}[thm]{Assumption}
\newtheorem{expectation}[thm]{Expectation}

\DeclareMathOperator{\hP}{\mathbb{P}}
\DeclareMathOperator{\ord}{ord}

\DeclareMathOperator{\bC}{\mathbb{C}}
\DeclareMathOperator{\C}{\mathbb{C}}

\DeclareMathOperator{\bN}{\mathbb{N}}
\DeclareMathOperator{\N}{\mathbb{N}}

\DeclareMathOperator{\bP}{\mathbb{P}}
\DeclareMathOperator{\bQ}{\mathbb{Q}}
\DeclareMathOperator{\Q}{\mathbb{Q}}
\DeclareMathOperator{\bR}{\mathbb{R}}

\DeclareMathOperator{\bZ}{\mathbb{Z}}
\DeclareMathOperator{\Z}{\mathbb{Z}}

\newcommand{\cE}{\mathcal{E}}
\newcommand{\cF}{\mathcal{F}}

\newcommand{\cH}{\mathcal{H}}

\newcommand{\cL}{\mathcal{L}}

\newcommand{\cT}{\mathcal{T}}

\newcommand{\cV}{\mathcal{V}}

\title{On periods of Elliptic curves}

\author{Daniel Barrera Salazar}
\address{Universidad de Santiago de Chile, Alameda 3363, Santiago, Chile.}
\email{daniel.barrera.s@usach.com}

\author{Juan-Pablo Llerena-C\'ordova}
\address{Department of Mathematics and Statistics, University of Ottawa, 150 Louis-Pasteur Pvt, Ottawa, ON, Canada K1N 6N5}
\email{jller032@uottawa.ca}

\begin{document}
\begin{abstract}
Let $E$ be an elliptic curve over $\Q$ having split multiplicative reduction at a prime number $p$. We describe the tame part of the $\cL$-invariant of $E$ at $p$ in terms of automorphic $p$-adic periods introduced in the work of Darmon. More precisely, we prove an equality of refined $\cL$-invariants using twisted versions of refined exceptional zero conjectures. When the conductor of the elliptic curve is exactly $p$ and the automorphic period is attached to an optimal embedding of conductor $1$ then we prove this equality unconditionally by using the work of de-Shalit. 

\end{abstract}

\maketitle


\section{Introduction}
In the 1980s, Mazur, Tate, and Teitelbaum formulated $p$-adic analogs of the Birch and Swinnerton-Dyer (BSD) conjectures in \cite{mazur-tate-teitelbaum}. 
In the same decade, a different variation of the BSD conjectures appeared; \emph{refined} analogs of the BSD conjectures were stated by Mazur and Tate in \cite{MT87}. In these two versions a new phenomenon was observed: the \emph{existence of exceptional zeros}. On the $p$-adic side, \emph{the exceptional zero conjecture} was proven by Greenberg and Stevens in \cite{greenberg-stevens} using $p$-adic deformation. Surprisingly, in the refined setting, the corresponding conjecture was partially proven by de-Shalit in \cite{deShalit} following the same $p$-adic procedure that Greenberg and Stevens used in \cite{greenberg-stevens}. The present work is an attempt to pursue the transposition of techniques used to study elliptic curves from the $p$-adic setting to the refined setting.  

Let $E$ be an elliptic curve over $\Q$ having split multiplicative reduction at a prime number $p$. By the work of Tate \cite[Theorem 5]{Tat74}, we know that there exists a unique $q_E \in \Q_p^\times$ such that $\ord_p(q_E) > 0$ and $\C_p^{\times}/q_E^{\Z} \rightarrow E(\C_p)$ as rigid analytic spaces, where $\ord_p$ is the $p$-adic valuation normalized by $\ord_p(p) = 1$. The $p$-adic number $q_E$,  called the \textit{$p$-adic Tate period} attached to $E$, is the main object of interest in this article. 

Using the natural decomposition
\begin{equation}\label{period_decomposition}
	\Q_p^\times \cong p^{\bZ} \times \mu_{p-1} \times (1 + p\bZ_p),
\end{equation}
we can study the $p$-adic Tate period by its projection onto each factor of the decomposition of Equation \eqref{period_decomposition}. On the one hand, the projection to the first component of the decomposition of Equation \eqref{period_decomposition}, \textit{i.e.}, the $p$-adic valuation of $q_E$, is well-known to satisfy $\ord_p(q_E)= - \ord_p(j(E))$ (see \cite[Chapter II, \S1]{mazur-tate-teitelbaum}). On the other hand, the projection of $q_E$ into $1+ p\Z_p$ is the subject of different works, such as \cite{greenberg-stevens} and \cite{Dar01} for example. More precisely, in \cite{Dar01} the author attached an alternative $p$-adic period $I_{\Psi}\in \Q_p^{\times}$ to the elliptic curve $E$, using automorphic methods, which depends on a certain measure on $\bP^1(\Q_p)$ attached to $E$ and an optimal embedding $\Psi:\Q \times \Q \rightarrow M_2(\Q)$ (see Section \S\ref{s: automorphic periods} for the construction of $I_\Psi$). Furthermore, the author proved that these two periods are related 
\begin{thm}[{\cite[Theorem 1]{Dar01}}]\label{Result Darmon}
    We have the following equality
    \begin{equation}\label{e:formula de darmon}
    \log_p(I_{\Psi}) = \frac{\log_p(q_E)}{\ord_p(q_E)} \ord_p(I_{\Psi}),
    \end{equation}
\end{thm}

where $\log_p$ is the $p$-adic logarithm with $\log_p(p) = 0$. Moreover, $I_\Psi$ is conjecturally a power of $q_E$, see Conjecture 3 in \cite{Dar01}.

Equation \eqref{e:formula de darmon} does not detect the information regarding the projection of $q_E$ to $\mu_{p-1}$, the group of $(p-1)$-th roots of unity. 
Our main goal is to prove an analogous formula to Equation \eqref{e:formula de darmon} by replacing $\log_p$ and $\ord_p$ by functions that capture this information.

Fix $\ell$ a prime number dividing $p-1$ and let $\ell^{m}$ be the maximal power of $\ell$ dividing $p-1$. We denote by $R= \Z/\ell^m\Z$ and fix a logarithm $\mathrm{log}: (\Z/p\Z)^{\times}\rightarrow R$ i.e. a surjective group homomorphism. We denote by $v_R, \lambda_R: \Q_p^{\times}\rightarrow R$ the homomorphisms defined as follows: if $q\in \Q_p^{\times}$, then we let $v_R(q)= \mathrm{ord}_p(q) \ (\mathrm{mod} \ \ell^m)$ and $\lambda_R(q)= \mathrm{log}(\tilde{q} \ (\mathrm{mod} \ p))$ where $\tilde{q}= q\cdot p^{-\mathrm{ord}_p(q)}$. In particular, we obtain an identification 
$\Q_p^{\times}\otimes_{\Z}R \cong R^2, \ \ q\mapsto (v_R(q), \lambda_R(q))$. In this article, we prove the following.
\begin{thm}\label{theorem_1} Assume that $E$ is an elliptic curve defined over $\Q$ has prime conductor $p$. Let $\ell> 3$ be a prime number coprime to the modular degree of $E$. If $\Psi$ is an optimal embedding of conductor $1$, then:
\begin{equation}\label{eqtheorem_1}
\lambda_{R}(I_{\Psi})= \frac{\lambda_{R}(q_E)}{v_R(q_E)}\cdot v_R(I_{\Psi}).
\end{equation}
\end{thm}
As explained in \cite[\S 6.3]{deShalit}, the hypothesis on the degree of the parametrization implies that $v_R(q_E) \in R^{\times}$. Thus, Equation \eqref{eqtheorem_1} is well-defined.  To prove our theorem, we use a refined version of the strategy carried out in the proof of Theorem \ref{Result Darmon}. The strategy in \cite{Dar01} used, in a crucial step, the $p$-adic exceptional zero conjecture proved by Greenberg-Stevens in \cite{greenberg-stevens}. As already mentioned in \cite{deShalit}, de-Shalit proved a refined analog of the Greenberg-Stevens result and, under certain hypotheses, the so-called \emph{refined conjecture} stated by Mazur and Tate in \cite[p.712]{MT87}. To prove our theorem, we follow the strategy carried out by Darmon in \cite[\S2]{Dar01} in our refined situation and we use the main result of de-Shalit.

To prove the theorem above, we use the strong assumptions that the conductor of $E$ is $p$ and that the conductor of $\Psi$ is $1$. Nevertheless, several calculations are valid when $E$ has arbitrary conductor and the conductor of $\Psi$ is coprime to the conductor of $E$. In Theorem \ref{teorema_fuerte} we assume Conjecture \ref{conjetura twisteada}, which is a generalization of the \emph{refined conjecture} stated by Mazur and Tate, and a refined version of \cite[ \S 13 Conjecture 1]{mazur-tate-teitelbaum}, this allows us to remove the hypothesis that $E$ has prime conductor and that the optimal embedding has conductor $1$ in Theorem \ref{theorem_1}. We hope that this conjecture can be studied by generalizing de-Shalit's approach.
\begin{rmk}
    Recently in  \cite{BH25} Bullach and Honnor announced a proof of Conjecture 6 of \cite{MT87} (which implies Conjecture \ref{conjetura twisteada}, see Lemma \ref{lem:5.5}) in more generality than in \cite{deShalit}. Their approach is different from the approach of de-Shalit. Indeed, Bullach and Honnor use a relation between Selmer complexes and Euler systems developed by Bullach and Burns in \cite{BP25}. Applying the result of Bullach and Honnor, we could replace the assumption that $E$ has prime conductor and that $\ell$ is coprime to the modular degree, with the assumption that $E$ is semistable, does not have CM, and $\ell \geq 11$ (see \cite[Corollary 1.9]{BH25}).
\end{rmk}
In Section \S\ref{s: Modular Symbols} and Section \S\ref{s: automorphic periods} we summarize the necessary notation from \cite{Dar01} and define the period $I_{\Psi}$. In Section \S\ref{s: Formulae}, we transpose the computations from \cite{Dar01} to obtain formulae for $\lambda_R(I_{\Psi})$ and $v_R(I_{\Psi})$ in terms of modular symbols. In Section \S\ref{A refined version of exceptional zero conjectures}, we state twisted versions of the conjectures stated by Mazur and Tate. In \S\ref{Periods and the Exceptional zero phenomena}, we prove our main theorem. Finally, in Section \S\ref{s: j-invariant}, we present numerical support for an expectation on the $p$-adic valuation of the $j$-invariant which is needed to define the refined $\cL$-invariant in general.

\subsection*{Acknowledgements} We thank H. Darmon for his support and motivation. We also thank V. Rotger and M. Masdeu for their interest in this work. Finally, we would like to thank A. Lei for feedback on an earlier version of this article. This article started as part of JPLC master's thesis at the University of Chile and funded by ANID scholarship Mag\'ister Nacional 2021 N 22221372. DBS was supported by ECOS230025 and  ANID FONDECYT grant 1241702.


\section{Modular Symbols}\label{s: Modular Symbols}

Let $E/\bQ$ be an elliptic curve over the field of rational numbers, and let $N\in \bN$ be the conductor of $E$. Unless stated otherwise, we will make the following assumption for the rest of the article.
\begin{assumption}\label{Q_isogeny}
	The elliptic curve $E$ is the optimal elliptic curve\footnote{Alternatively referred to as the strong Weil curve.} in its $\Q$-isogeny class.
\end{assumption}
We will prove that this assumption is not essential for Theorem \ref{theorem_1} (see Lemma \ref{lemma_assumption}). Nevertheless, this assumption simplifies the proof of Theorem \ref{theorem_1}. Thus, we will prove Theorem \ref{theorem_1} under this assumption; the general case will be a consequence of Lemma \ref{lemma_assumption}).

\begin{rmk}
	Note that in \cite{Dar01}, Darmon makes the stronger assumption that $E$ is unique in its $\bQ$-isogeny class (see \cite[Assumption 2]{Dar01}). By the Modularity Theorem, we know that every $\bQ$-isogeny class contains an optimal curve. So, Assumption \ref{Q_isogeny} is weaker than assuming that $E$ is unique in its $\Q$-isogeny class. 
    
    The stronger assumption that $E$ is unique in its $\bQ$-isogeny class is essential in \cite{Dar01}, as some of the conjectures appear to not hold if $E$ is not unique in its $\bQ$-isogeny class (see Conjecture $5$ in \cite{Dar01}). However, as pointed out by Darmon, the assumption that $E$ is unique in its $\bQ$-isogeny class is not essential for the proof of Theorem \ref{Result Darmon} (see the discussion following Conjecture $5$ of \cite{Dar01}), and since we are only interested in a refined version of such theorem, Assumption \ref{Q_isogeny} is sufficient.
\end{rmk}
Denote by $\omega_E$ a N\'eron differential of $E$ and by $\Lambda_E := \left\{\int_{[\gamma]} \omega_E; \ \  [\gamma]\in H_1(E(\bC), \bZ)\right\}$ its N\'eron lattice. Also, let $\Omega_E^{\pm}\in \bR$ be the unique positive real numbers such that one of the following conditions holds:
\begin{itemize}
	\item $\Lambda_E = \Omega_E^+ \cdot \bZ + \Omega_E^- \bZ i$ \hspace{7pt}(Rectangular case).
	\item $\Lambda_E \subseteq \Omega_E^+ \cdot \bZ + \Omega_E^- \bZ i$ and $\Lambda_E$ consists of the elements of the form $a\Omega_E^+ + b \Omega_E^- i$ where $a\equiv b \pmod{2}$\hspace{7pt} (Non-rectangular case).
\end{itemize}
Note that $\Omega_E^+ = \frac{1}{2}\int_{E(\bR)} |\omega_E|$. Thus, $\Omega_E^+$ is equal to the least positive real period of $\Lambda_E$ in the rectangular case, and $\Omega_E^+$ is equal to half the least positive real period of $\Lambda_E$ in the non-rectangular case.

Denote by $\cH$ the upper-half plane, and let $\cH^* = \cH \cup \bP^1(\bQ)$ denote the extended upper-half plane. By the Modularity Theorem, we know that there exists a complex parametrization $\phi: \Gamma_0(N)\setminus \cH^* \rightarrow E(\bC)$, such that
\[
	\phi^*(\omega_E) = c_E2\pi i f_E(z) dz,
\]
where $f_E\in S_2(\Gamma_0(N))$ is a newform and $c_E\in \Z$ is the Manin constant, which is well-defined up to sign (for a survey regarding the Manin constant see \cite{ARS06}).

By the Theorem of Manin-Drinfeld (\cite{Drinfeld}, \cite{Manin}) there exist rational numbers $\modsymb{a}{b}^{\pm}_{\mathrm{raw}}\in \bQ$ such that
\[
    2\pi i \int_{i\infty}^{-\frac{a}{b}} f_E(z) dz= \modsymb{a}{b}^+_{\mathrm{raw}} \cdot \Omega_E^+ + \modsymb{a}{b}^-_{\mathrm{raw}} \cdot \Omega_E^-i.
\]
Furthermore, the Theorem of Manin-Drinfeld guarantees that there exists a $D\in \bN$ such that, for all $\frac{a}{b}\in \bQ$ we have that
\begin{equation}\label{modsymb}
	\modsymb{a}{b}^{\pm} := D\cdot\modsymb{a}{b}^{\pm}_{\text{raw}}\in \bZ,
\end{equation}

that is, the denominators of $\modsymb{a}{b}^{\pm}$ are bounded. Now, fix $D$ to be the smallest natural number that satisfies Equation \eqref{modsymb} for all $\frac{a}{b}\in \bQ$. We refer to the values $\modsymb{a}{b}^{\pm}$ as the ``$+$'' or ``$-$'' modular symbols of $E$ at $\frac{a}{b}\in \bQ$. The value $D$ is of great interest, as it would allow for better optimizations when calculating modular symbols (See \cite[Section 2]{Wut18}). Under Assumption \ref{Q_isogeny} and assuming that $N$ is square-free, it is known that $\frac{1}{D}\in \bZ\left[\frac{1}{c_E}, \frac{1}{2}, \frac{1}{\#E(\bQ)_{tor}}\right]$ (See \cite[Proposition 1]{Wut18}). For the general case, see \cite[Theorem 1]{Drinfeld}. Nonetheless, if we do not consider the smallest natural number $D$ that satisfies Equation \eqref{modsymb} (\textit{i.e.}, we consider a larger value $D$), then Theorem \ref{theorem_1} still holds.

Since we are only interested in the ``$+$'' modular symbol, we will denote the ``$+$'' modular symbol attached to $E$ as $\modsymb{a}{b} :=  \modsymb{a}{b}^+\in \Z$ (\textit{cf.} \cite[\S 1.2]{Dar01}). The values $\modsymb{a}{b}$ will be key to defining a measure on $\mathbb{P}_1(\mathbb{Q}_p)$.

\begin{rmk}\label{remark_multiplication}
	This slight variation on the definition of modular symbols can be thought of as multiplying the modular symbols as defined in \cite[\S1.2]{Dar01} by a integer to guarantee that the modular symbols are integers. The inclusion of the variable $D$ in Equation \eqref{modsymb} is such that this ``scaling process'' is independent of the choice of $\frac{a}{b}\in \Q$. See Remark \ref{importantremark} for the consequence of considering a different value $D$ in Equation \eqref{modsymb}.
\end{rmk}

\begin{rmk}
	This definition of the ``$+$'' symbol differs from the convention used in \cite{Dar01}, where the integral of Equation \eqref{modsymb} is multiplied by $2\pi i c_E$ rather than $2\pi i$ (we are using the convention used in \cite[\S 6.1.c]{deShalit}). The reason for the convention in \cite{Dar01} is to guarantee that the modular symbols are integers, which is necessary for the construction of the measure on $\mathbb{P}_1(\mathbb{Q}_p)$. However, in our case, we ``scale'' the modular symbols to guarantee that they are integers.
    
	Under Assumption \ref{Q_isogeny}, it is expected that $c_E = 1$ (See \cite{ARS06}). Therefore, conjecturally, the convention made in \cite{Dar01} and Equation \eqref{modsymb} are expected to be equal. Nevertheless, if $c_E \ne 1$, one can multiply Equation \eqref{modsymb} by $c_E$ and it will not affect the proof of Theorem \ref{theorem_1}. But it may trivialize Theorem \ref{theorem_1} (see Remark \ref{importantremark}).
\end{rmk}

\begin{rmk}\label{prop_denominadores} In fact, results from Mazur and Wuthrich give information about the number $D$. Suppose that $E$ has multiplicative reduction at $p$ and let $c\in \bN$ be coprime to the conductor of $E$ then we have:
	\begin{enumerate}
		\item If $\ell > 7$, then the denominator of $\modsymb{a}{p^nc}_{\text{raw}}$ is coprime to $\ell$ for all $a\in (\bZ/p^nc\bZ)^\times$.
		\item If $E$ has prime conductor $p\ne 11$ and $\ell > 3$, then the denominator of $\modsymb{a}{p^nc}_{\text{raw}}$ is coprime to $\ell$ for all $a\in (\bZ/p^nc\bZ)^\times.$
	\end{enumerate}

Indeed, since $E$ has multiplicative reduction at $p$ and $c$ is coprime to the conductor of $E$, then by \cite[Proposition 1]{Wut18} we have that $\modsymb{a}{p^nc}_{\text{raw}} \in \frac{1}{2\#E(\bQ)_{\mathrm{tor}}}\bZ$. Then (1) is a consequence of the classification of the torsion of $E(\bQ)$  in \cite[III.5, Theorem 5.1]{Mar77}. As we know that $E(\bQ)$ cannot have a point of torsion $\ell > 7$, then the denominator of $\modsymb{a}{p^nc}_{\text{raw}}$ is coprime to $\ell$.

As pointed out by Mazur \cite[III.7]{Mar77}, if $E$ has prime conductor, then it $E$ can only have a torsion point of order $\ell \geq 5$. Moreover, in \cite{Miy73} it is proven that if $E(\bQ)$ has prime conductor and has a torsion point of order $\ell = 5$, then the conductor of $E$ is $11$. Thus we obtain (2).
\end{rmk}

\section{Automorphic periods}\label{s: automorphic periods}

\subsection{Measures on $\mathbb{P}_1(\mathbb{Q}_p)$}\label{ss: Measure}
From now on, we assume that $E$ has split multiplicative reduction at  $p$, and we write $N= pM$ with $M\in \N$ and $p\nmid M$. 

Let $G := \{(\begin{smallmatrix}
		a & b\\
		c & d
	\end{smallmatrix})\in \mathrm{M}_2(\bZ[1/p]); \hspace{3pt} c \equiv 0 \pmod{M}\}$. We denote by $\Gamma \subseteq \text{PSL}_2(\mathbb{Z}[1/p])$ the image under the natural projection $\mathrm{SL}_2(\mathbb{Z}[1/p]) \twoheadrightarrow \text{PSL}_2(\mathbb{Z}[1/p])$ of the matrices $\gamma \in G$ such that $\mathrm{det}(\gamma) = 1$. We also denote by $\tilde{\Gamma}\subset \text{PGL}_2(\mathbb{Z}[1/p])$ the image under the natural map $\mathrm{GL}_2(\mathbb{Z}[1/p]) \twoheadrightarrow \mathrm{PGL}_2(\bQ_p)$ of the matrices $\gamma\in G$ such that $\mathrm{det}(\gamma) > 0$. 

We denote by $\cT$ the Bruhat-Tits tree of $\mathrm{PGL}_2(\Q_p)$. We will also denote by $\cV(\cT)$ the set of vertices of $\cT$ (equivalence classes of rank $2$ $\mathbb{Z}_p$-free modules modulo $\mathbb{Q}_p^{\times}$-homothety) and by $\cE(\cT)$ the set of ordered edges of $\cT$ (two vertices are connected if there exist representatives such that one contains the other with index $p$). Furthermore, let $v_*\in \cV(\cT)$ be the vertex whose equivalence class contains $\Z_p^2$, and let $e_*$ be the edge with source $v_*$ and with target the vertex whose equivalence class contains $\Z_p\times p\Z_p$.

We can identify the elements of $\cE(\cT)$ with basic compact open subsets of $\hP_1(\Q_p)$ as follows: Set $U(e_*) := \Z_p\subseteq \hP_1(\Q_p)$. Now, for an arbitrary edge $e\in\cE(\cT)$, let $\gamma\in \mathrm{GL}_2(\Q_p)$ be such that $e = \gamma e_*$, and set 
\[
	U(e) := \gamma(U(e_*)) = \{x\in \hP(\Q_p);\hspace{3pt} \gamma^{-1}x\in \Z_p\}.
\]
Fix $x,y\in \mathbb{P}_1(\Q_p)$. From \cite[Sec. 1.2]{Dar01} we obtain a measure $\mu\{x\rightarrow y\}$ on $\mathbb{P}_1(\Q_p)$ such that for each $e \in\cE(\cT)$, we have
\[
    \mu\{x\rightarrow y\}\left(U(e)\right) := \modsymb{c}{d} - \modsymb{a}{b}.
\]
where $\frac{a}{b}, \frac{c}{d}\in \Q$ are such that $\gamma x = -\frac{a}{b}$, $\gamma y = -\frac{c}{d}$, and $\gamma\in \tilde{\Gamma}$ is such that $\gamma e = e_*$.

\subsection{The period $I_{\Psi}$} \label{ss:automorphic period}

Let $c\in \bN$ be coprime to $N$. Also, let $\Psi: \Q\times \Q \rightarrow M_2(\Q)$ be an optimal embedding of conductor $c$, i.e., a $\Q$-algebra embedding such that
\[
\Psi^{-1}\left(\left\{\begin{pmatrix}
				a & b\\
				c & d
				\end{pmatrix}\in M_{2}\left(\bZ[1/p]\right); c \equiv 0 \pmod{M}
\right\}\right) = \{(u,v)\in \mathbb{Z}[1/p] \times \mathbb{Z}[1/p]; \hspace{3pt} u \equiv v\ (\mathrm{mod}\ c)\}.
\]
If $\nu\in\bN$ with $(\nu, c) = 1$ we have an optimal embedding $\Psi_{\nu}: \Q\times \Q \rightarrow M_2(\Q)$ defined via
\begin{equation}\label{eq:optimal embedding}
    \Psi_{\nu}(a,a) := \begin{pmatrix} a & 0\\ 0 & a \end{pmatrix} \text{ and }\Psi_{\nu}(c,0) := \begin{pmatrix} c & \nu\\ 0 & 0 \end{pmatrix}.
\end{equation}
Now, given an optimal embedding $\Psi$, denote by $\overline{\Psi}: \Q^{\times}\times \Q^{\times} \rightarrow \mathrm{PSL}_2(\Q)$ the induced group homomorphism by $\Psi$. We have that $\overline{\Psi}(\Q^{\times}\times \Q^{\times})$ acts on $\cH^*$ (via M\"obius transformations) with two fix points $x_{\Psi},y_{\Psi}\in\hP_1(\Q)$, and the group $\overline{\Psi}(\Q^{\times}\times \Q^{\times}) \cap \Gamma$ has rank $1$. We fix a generator $\gamma_{\Psi} \in \overline{\Psi}(\Q^{\times}\times \Q^{\times}) \cap \Gamma$ and we define the automorphic period by 
\[
	I_{\Psi} := \mint_{\mathbb{P}_1(\Q_p)} \left(\frac{t - \gamma_{\Psi}z}{t-z}\right) \mu\{x_{\Psi} \rightarrow y_{\Psi}\}(t) \in \C_p^{\times}.
\]
Here $z\in \cH_p$ is any fixed point, and we consider the multiplicative integral (See \cite[Equation (72)]{Dar01}).  In \cite[Lem. 2.3 and Prop. 2.7]{Dar01}, Darmon proved that $I_{\Psi}$ is independent of the choice of $z\in \cH_p$ and $I_{\Psi}\in \Q_p^{\times}$. Finally, by interchanging $x_{\Psi}$ and $y_{\Psi}$ we can assume that $x_{\Psi}$ is an attractive point and $y_{\Psi}$ is a repulsive point of $\gamma_{\Psi}$.

\section{Formulae for $\lambda_R(I_{\Psi})$ and $v_R(I_{\Psi})$}\label{s: Formulae}
From now on, we will make the following assumption on the optimal embedding $\Psi$ of conductor $c$.
\begin{assumption}\label{assumption:optimal embedding}
    The optimal embedding $\Psi$ of conductor $c$ has the form $\Psi_{\nu}$ for some $\nu\in \N$ coprime to $c$, where $\Psi_\nu$ is defined as in Equation \eqref{eq:optimal embedding}.
\end{assumption}

\begin{rmk}\label{simplify_rmk} To prove Theorem \ref{theorem_1}, it is enough prove it under Assumption \ref{assumption:optimal embedding}. Indeed, by \cite[Lemmas 2.12, 2.2 and 2.3]{Dar01} there exists a $\nu\in \N$ coprime to $c$ such that $I_{\Psi_{\nu}} = I_{\Psi}^{\pm 1}$. Moreover, the image of $\nu$ in $(\Z/c\Z)^{\times}/\left<p^2\right>$ is uniquely determined by $\Psi$. Finally, observe that Theorem \ref{theorem_1} is invariant when replacing $I_{\Psi}$ by $I_{\Psi}^{-1}$. Thus, we can focus only on the case when the automorphic period is of the form $I_{\Psi_{\nu}}$ for some $\Psi_\nu$
\end{rmk}
Fix a logarithm $\mathrm{log}: (\Z/p\Z)^{\times}\rightarrow R$ i.e., a surjective group morphism, and denote by  $v_R, \lambda_R: \Q_p^{\times}\rightarrow R$ the homomorphisms defined by $v_R(q)= \mathrm{ord}_p(q) \ (\mathrm{mod} \ \ell^m)$ and $\lambda_R(q)= \mathrm{log}(\tilde{q} \ (\mathrm{mod} \ p))$ if $q\in \Q_p^{\times}$, where $\tilde{q}= q\cdot p^{-\mathrm{ord}_p(q)}$.

A first key step in Darmon's proof of Theorem 1 of \cite{Dar01}, in the $p$-adic context, are certain formulae relating $\mathrm{ord}_p(I_{\Psi})$ and $\mathrm{log}_p(I_{\Psi})$ to modular symbols. Below, we obtain analogous formulas in our refined situation; the line of argument is similar to \cite{Dar01}, where we replace $\log_p$ and $\ord_p$ with $\lambda_R$ and $v_R$, respectively.

\begin{prop}\label{ordPsi} Let $J$ be the subset of $(\mathbb{Z}/c\mathbb{Z})^{\times}$ consisting of the cosets $a$ such that $a/\nu \equiv p^j\ (\mathrm{mod}\ c)$ for some $j \in \Z$ depending on $a$. Then
\begin{equation}\label{eq:ordPsi}
    v_R(I_{\Psi}) = \sum_{a\in J} \modsymb{a}{c} \ \ (\mathrm{mod}\ \ell^m).
\end{equation}
\end{prop}
\begin{proof}
    By \cite[Proposition 2.13]{Dar01} we have that
    \begin{equation}\label{temp equation}
        \mathrm{ord}_p(I_{\Psi_\nu}) = \sum_{a\in J}\modsymb{a}{c}.
    \end{equation}
    Thus, to conclude the proof it is sufficient to consider modulo $\ell^m$ Equation \eqref{temp equation}.
\end{proof}

\begin{rmk} Our convention of normalization of the modular symbols $\modsymb{a}{c}$, is to guarantee that $\modsymb{a}{c}\in \Z$, as mentioned in Remark \ref{remark_multiplication}. Thus, Equation \eqref{eq:ordPsi} is well-defined.
\end{rmk}

We deduce a refined version  of the description of $\mathrm{log}_p(I_{\Psi})$ in terms of modular symbols.

\begin{prop}\label{logPsi-nueva} Assume that  $c\equiv 1 \pmod{p}$ and denote by $J_1$ the subset of $(\mathbb{Z}/pc\mathbb{Z})^{\times}$ consisting of the cosets $a$ such that $a/\nu \equiv p^j\ (\mathrm{mod}\ c)$ for some $j \in \Z$ depending on $a$. Then 
\begin{equation}\label{eq:logPsi-nueva}
	\lambda_R(I_{\Psi}) =  \sum_{a\in J_1}\lambda_R(a) \modsymb{a}{pc}.
\end{equation}
\end{prop}

\begin{proof} We adapt the calculations performed in \cite[Proposition 2.18]{Dar01} to our tame setting. First, observe that the function $\lambda_R: \Q_p^{\times}\rightarrow R$ is continuous, where in $\Q_p^{\times}$ we consider the $p$-adic topology and in $R$ the discrete topology.

We denote by $\cF_{\Psi}$ the fundamental domain of the action of $\gamma_{\Psi}$ on $\hP_1(\Q_p) - \{x_{\Psi}, y_{\Psi}\}$ defined in the discussion following Lemma 2.6 in \cite{Dar01}. A description of $\cF_{\Psi}$ is as follows: for a given $n\in\N$, $\cF_{\Psi}$ can be written as the following disjoint union (see the proof of \cite[Lemma 2.18]{Dar01})
\begin{equation}\label{fundamental_domain_decomposition}
	\cF_{\Psi} = \bigcup_{j=0}^{s-1}\left(\bigcup_{a\in \left(\Z/p^n\Z\right)^\times} U_{j,a}\right),
\end{equation}
where $s$ is twice the order of $p^2$ in $(\Z/c\Z)^{\times}$ and
\[
	U_{j,a} := \left\{t\in \bP^1(\Q_p) \mid \mathrm{ord}_p\left(t + \frac{\nu}{c}\right) = j \ \text{ and } \ p^{-j}\left(t +\frac{\nu}{c}\right)\equiv a \pmod{p^n}\right\}.
\]
From \cite[Equation (129)]{Dar01}, we have $\mu\{x_{\Psi}\rightarrow y_{\psi}\}\left(U_{j,a}\right) = \modsymb{a_j}{p^nc}$ where $a_j\in \bZ$ is any integer such that $a_j \equiv ac \pmod{p^n}$ and $a_j \equiv \nu p^{-j} \pmod{c}$. 

Recall that we let $x_{\Psi}$ and $y_{\Psi}$ denote the attractive and repulsive points of $\gamma_\Psi$, respectively. Additionally, using the fact that $\Psi$ is of the form $\Psi_\nu$, we conclude that (see \cite[Proposition 2.7]{Dar01})  
\[
	I_{\Psi} = p^s \cdot \mint_{\cF_{\Psi}} \left(t + \frac{\nu}{c} \right) d\mu\{x_{\Psi}\rightarrow y_{\psi}\}(t).
\]
If we set $\mu := \mu\{x_{\Psi}\rightarrow y_{\psi}\}$, we get (recall that $\lambda_R(p)= 0$)
\[
    \lambda_R(I_{\Psi}) = \lambda_R\left(\mint_{\cF_{\Psi}} (t + \tfrac{\nu}{c}) d\mu(t)\right) = \lambda_R\left(\lim_{n\rightarrow \infty} \prod_{j=0}^{s-1} \prod_{a\in \left(\bZ/p^n\bZ\right)^\times} \left(t_{j,a} + \frac{\nu}{c}\right)^{\mu(U_{j,a})}\right),
\]
where $t_{j,a}\in U_{j,a}$. Using the fact that $\lambda_R: \Q_p^{\times}\rightarrow R$ is continuous and $\mu\left(U_{j,a}\right) = \modsymb{a_j}{p^nc}$, we deduce
\begin{align*}
	\lambda_R\left(\lim_{n\rightarrow \infty} \prod_{j=0}^{s-1} \prod_{a\in \left(\bZ/p^n\bZ\right)^\times} \left(t_{j,a} + \frac{\nu}{c}\right)^{\mu(U_{j,a})}\right) =  \lim_{n\rightarrow \infty}\sum_{j=0}^{s-1} \sum_{a\in \left(\bZ/p^n\bZ\right)^\times} \modsymb{a_j}{p^nc} \lambda_R\left(t_{j,a} + \frac{\nu}{c}\right).
\end{align*}
As $t_{j,a}\in U_{j,a}$, we have $p^{-j}\left(t_{j,a} + \frac{\nu}{c}\right) \equiv a \pmod{p}$. Moreover, we know that $a_j \equiv ac \pmod{p^n}$ and by our hypothesis on $c$, we obtain, by the definition of $\lambda_R$, that $\lambda_R\left(t_{j,a} + \frac{\nu}{c}\right) = \lambda_R(a)= \lambda_R(ac)= \lambda_R(a_j)$. Thus, we get
\[ \lambda_R(I_{\Psi})= \lim_{n\rightarrow \infty}\sum_{j=0}^{s-1} \sum_{a\in \left(\bZ/p^n\bZ\right)^\times} \modsymb{a_j}{p^nc} \lambda_R(a_j)=  \lim_{n\rightarrow \infty} \sum_{a\in J_n} \modsymb{a}{p^nc} \lambda_R(a).\]
To conclude the proof of the proposition, we will prove that for each $n\geq 1$, we have:
$$\sum_{a\in J_n}\lambda_R(a)\modsymb{a}{p^nc}= \sum_{a\in J_{n+1}}\lambda_R(a) \modsymb{a}{p^{n+1}c}.$$

Firstly, observe that the natural morphism $\mathbb{Z}/p^{n+1}c\mathbb{Z}\rightarrow \mathbb{Z}/p^nc\mathbb{Z}$ induces a well-defined map $\pi: J_{n+1}\rightarrow J_n$. We have that if $a\in \mathbb{Z}/p^nc\mathbb{Z}$ then $\pi^{-1}(a)= \{a+ kp^n \mid \ k= 0, ..., p-1\}$. Note that as $a+ kp^n \equiv a (\mathrm{mod }\  p)$ then $\lambda_R(a+ kp^n)= \lambda_R(a)$. Moreover, from the fact that $E$ has split multiplicative reduction at $p$ and \cite[Ch. I, formula (4.2)]{mazur-tate-teitelbaum}, we have $\sum_{k= 0}^{p-1}\modsymb{a+ kp^n}{p^{n+1}c}= \modsymb{a}{p^{n}c}$.

Putting these observations together, we obtain 
$$\sum_{a'\in \pi^{-1}(a)}\lambda_R(a') \modsymb{a'}{p^{n+1}c} = \sum_{k= 0}^{p-1}\lambda_R(a+  kp^n)\modsymb{a+ kp^n}{p^{n+1}c}= \lambda_R(a) \sum_{k= 0}^{p-1}\modsymb{a+ kp^n}{p^{n+1}c}= \lambda_R(a) \modsymb{a}{p^{n}c}.$$

Thus, we have:  
$$\sum_{a'\in J_{n+1}}\lambda_R(a') \modsymb{a'}{p^{n+1}c}= \sum_{a\in J_n}\sum_{a'\in \pi^{-1}(a)}\lambda_R(a') \modsymb{a'}{p^{n+1}c}=  \sum_{a\in J_n}\lambda_R(a) \modsymb{a}{p^{n}c}.$$

\end{proof}
\section{A refined version of exceptional zero conjectures}\label{A refined version of exceptional zero conjectures}
In the same way as in the proof of Theorem \ref{Result Darmon}, our study of the tame part of $p$-adic periods of the elliptic curves is based on the study of the ``exceptional vanishing'' of the Mazur-Tate elements considered in \cite{MT87}. In particular, the so-called ``refined conjecture'' stated in \cite[Introduction pag. 712]{MT87}. This conjecture was partially proven by de-Shalit in \cite{deShalit} following the strategy of the work by Greenberg--Stevens in \cite{greenberg-stevens}. The conjectures about the exceptional zero phenomena in \cite{MT87} are stated in the setting of elliptic curves. In particular, no conjectures are made for more general modular forms or even general twisted elliptic curves. The corresponding $p$-adic version of the exceptional zero conjectures for general twisted elliptic curves are described in \cite[\S13]{mazur-tate-teitelbaum}. As we have not found a precise reference, in the following we will describe twisted exceptional conjectures in the refined context.

We use the same notation as in the previous sections.

\begin{con}\label{conjetura twisteada} Let $c\geq 1$ be an integer coprime to $p$. Also, let $\ell \geq 5$ be a prime number which is coprime to the modular degree of $E$. Then for each Dirichlet character $\chi: (\Z/c\Z)^{\times}\rightarrow \C_p^{\times}$ such that $\chi(p)= 1$ we have: 
\begin{equation}\label{equation_twisted}
   v_R(q_E) \left(\sum_{a\in (\mathbb{Z}/pc\mathbb{Z})^{\times}} \chi(a) \otimes \lambda_R(a)\modsymb{a}{pc}\right) = \lambda_{R}(q_E)\sum_{a\in (\mathbb{Z}/c\mathbb{Z})^{\times}} \left(\chi(a) \otimes \modsymb{a}{c}\right) \ \text{ in } \ \Z[\mathrm{Im}(\chi)] \otimes_{\bZ} R,
\end{equation}
where $\Z[\mathrm{Im}(\chi)]$ denotes the $\Z$-algebra generated by the image of $\chi$.
\end{con}

\begin{rmk}\label{importantremark}
    Recall the constant $D$ from Equation \eqref{modsymb}, if $D$ is coprime to $\ell$ then we can use the ``non-normalized modular symbols'' $\modsymb{a}{b}_{\mathrm{raw}}$ to state this conjecture and thus $D$ would not be necessary. Also, note that if instead of considering the smallest number $D$ on Equation \eqref{modsymb}, we consider a larger number $kD$ where $k\in \bN$, then that will be equivalent to multiplying Equation \eqref{equation_twisted} by $k$, and if $(k,\ell) > 1$, then this may trivialize the conjecture.
\end{rmk}

\begin{rmk} We comment on the hypotheses of this conjecture. 
\begin{enumerate}
\item The condition $\chi(p)= 1$ is needed to guarantee the existence of an exceptional zero as our elliptic curve has split multiplicative reduction at $p$. 
\item Even though the conditions that $\ell \geq 5$ and that $\ell$ is coprime to the modular degree of $E$ are present in de-Shalit's work when the conductor of $E$ is prime (see \cite[Theorem 0.3]{deShalit}), de-Shalit mentions that the coprimality with the modular degree is a "buit-in-flag" of his approach. Note that if we allow any $\ell\geq 5$ there are counter-examples to this conjecture (see \cite{LLe24}). However, in the counter-examples found therein, it appears that the condition that $\ell$ is coprime to the modular degree of $E$ is stronger that the condition that $(\ell, \#E(\bQ)_{\mathrm{Tor}}) = 1$. 

\item We did not mention any condition connecting $\ell$ and $c$. However, for relating Conjecture \ref{conjetura twisteada} with Theorem \ref{teorema_fuerte}, we will need $(\ell,\phi(c)) = 1$  where $\phi$ denotes the Euler-totient function. Nonetheless, we do not see any reasons to expect why this condition might be necessary in the conjecture above.
\end{enumerate}
\end{rmk}

We relate Conjecture \ref{conjetura twisteada} to the conjectures stated in \cite{MT87}. Recall that if $G$ is a finite abelian group and $A\subseteq \bQ$ a subring, the augmentation ideal is defined as $I(A,G) := \mathrm{ker}(\mathrm{aug})$, where $\mathrm{aug}:A[G] \rightarrow A; g\mapsto 1$. By \cite{LLe24}, there exists an isomorphism of $A$-modules
	\[
		I(A,G)/I(A,G)^2 \cong \bigoplus_{\substack{\ell \notin A^\times\\ \ell \mid |G|}} \mathrm{Syl}_\ell(G),
	\]
	where $\mathrm{Syl}_{\ell}(G)$ is the $\ell$-Sylow subgroup of $G$. If $G = (\bZ/p\bZ)^\times$ there is a natural isomorphism of $A$-modules
	\begin{equation}\label{eq_iso}
		I(A,G)/I(A,G)^2 \cong \bigoplus_{\substack{\ell \notin A^\times\\ \ell \mid p-1}} \bZ/\ell^{n_{\ell}}\bZ,
	\end{equation}
	where $n_{\ell}\in \bZ_{\geq 1}$ and $\ell^{n_{\ell}}|| p-1$. This is given by $\sum_{g\in G} \alpha_g g + I(A,G)^2 \mapsto \sum_{g\in G} \alpha_g g \pmod{\ell^{n_\ell}}$, for each $\ell$.

\begin{con}[{\textit{c.f.} \cite[Refined Conjecture]{MT87}}]\label{refined_conjecture} 
We have that
	\begin{equation}\label{refined_equation}
		\prod_{1 \leq a \leq p-1} a^{\mathrm{ord}_p(q_E)\modsymb{a}{p}D} \equiv \left(\frac{q_E}{p^{\mathrm{ord}_p(q_E)}}\right)^{\modsymb{0}{p}D} \pmod{p} \text{ in } \left(\bZ/p\bZ\right)^\times,
	\end{equation}
where $D$ is defined as in Equation \eqref{modsymb}
\end{con}
\begin{rmk}
    Note that the term $D$ in Equation \eqref{refined_equation} does not appear in \cite[Refined Conjecture, p. 712]{MT87}. However, if we do not include the factor $D$, then the modular symbols may have denominators that are not coprime to $p-1$. Thus, the terms appearing Equation \eqref{refined_equation} will not be defined. Therefore, to guarantee that Equation \eqref{refined_equation} is well-defined, we add the factor $D$ (this convention was also used by de-Shalit, see \cite[Introduction]{deShalit}).
\end{rmk}

Let $\ell\geq 5$ be a prime number, then Conjecture \ref{refined_conjecture} implies Conjecture \ref{conjetura twisteada} when $c=1$. In fact, it is enough to apply $\lambda_R$ to Equation \eqref{equation_twisted}.

\begin{con}[{\textit{c.f.} \cite[Conjecture 6]{MT87}}]\label{Conjecture_6} Let $p > 2$ and $G_p =: \left(\bZ/p\bZ\right)^\times/\left<-1\right>$. Let $A\subseteq \bQ$ be a subring such that $\left\{\modsymb{a}{p}\right\}_{0\leq a \leq p-1}\subseteq A$, $\frac{\modsymb{0}{1}}{2\mathrm{ord}_p(q_E)} \in A$ and $\frac{1}{\#E(\bQ)_{\mathrm{Tor}}}\in A$. Then $\sum_{1 \leq a \leq \frac{p-1}{2}}\modsymb{a}{p} a\in I(A,G_p)$ and
	\begin{equation}\label{equacion_con_6}
		\frac{1}{2}\sum_{1 \leq a \leq p-1}\modsymb{a}{p} a \equiv \frac{\modsymb{0}{1}}{2\mathrm{ord}_p(q_E)} (\tilde{q}_E - 1) \text{ in } I(A,G_p)/I(A,G_p)^2.
	\end{equation}
\end{con}
\begin{rmk}
    In the original statement of Conjecture 6 of \cite{MT87}, the condition that $\frac{1}{\#E(\bQ)}\in A$ is not present. However, it was shown in \cite{LLe24} that this condition is necessary for Conjecture \ref{Conjecture_6}.
\end{rmk}
\begin{lem}\label{lem:5.5}
    Let $A\subseteq \bQ$ be a subring defined as in Conjecture \ref{Conjecture_6}. Then Conjecture \ref{Conjecture_6} implies Conjecture \ref{conjetura twisteada} for $c = 1$ and all primes $\ell\notin A^\times$.
\end{lem}
\begin{proof}
    By Equation \eqref{eq_iso} we have that $I(A,G_p)/I(A,G_p)^2 \cong \bigoplus_{\substack{\ell|\frac{p-1}{2}\\ \ell \notin A^\times}} \bZ/\ell^{n_{\ell}}\bZ$. Remark that $\mathrm{Syl}_{\ell}(G_p) \cong \mathrm{Syl}_{\ell}\left(\left(\bZ/p\bZ\right)^\times\right)$ for each prime $\ell > 2$. Now, consider the natural projection $\pi_{\ell}:\bigoplus_{\ell|\frac{p-1}{2}; \ell \notin A^\times} \bZ/\ell^{n_{\ell}}\bZ \rightarrow \bZ/\ell^{n_{\ell}}\bZ$. Then, by composing the isomorphism of equation \eqref{eq_iso} and $\pi_{\ell}$ in Equation \eqref{equacion_con_6} we obtain
    \begin{equation}\label{temp:equation}
        \frac{1}{2}\sum_{1 \leq a \leq p-1} \mathrm{ord}_p(q_E)\modsymb{a}{p} a \equiv \lambda_R(\tilde{q}_E) \frac{\modsymb{0}{1}}{2} \pmod{\ell^{n_{\ell}}}.
    \end{equation}
    Reordering Equation \eqref{temp:equation} we obtain
    \[
        2\sum_{1 \leq a \leq p-1} \mathrm{ord}_p(q_E)\modsymb{a}{p} \lambda_R(a) \equiv 2\lambda_R(\tilde{q}_E)\modsymb{0}{1} \pmod{\ell^{n_{\ell}}}.
    \]
    Finally, because $\ell > 2$ then $2$ is invertible. Again, noting that the only character is the trivial one, we conclude the proof.
\end{proof}

\begin{lem}\label{lemma_assumption}
	Equation \eqref{equation_twisted} is invariant under $\Q$-isogenies of degree coprime to $\ell$.
\end{lem}
\begin{proof} Let $E_1/\Q$ be an elliptic curve of conductor $N$ with split multiplicative reduction at $p$ and  let $E_2/\Q$ be an elliptic curve such that there is $\Q$-isogeny to $E_1$ of degree coprime to $\ell$. In particular, $E_2$ also has split multiplicative reduction at $p$ as both elliptic curves are associated to the same modular form. Denote by $q_{E_1}$ and $q_{E_2}$ the Tate $p$-adic periods of $E_1$ and $E_2$ respectively. Then there exist $n,m\in \N$ such that $q_{E_1}^n = q_{E_2}^m$ with $(n,\ell) = (m, \ell) = 1$ (\cite[Chapter V, Exercise 5.10]{Sil94}). In particular, $q_{E_1} = \zeta q_{E_2}^{\frac{m}{n}}$for some $n$-root of unity $\zeta$. We obtain the following relations
	\begin{equation}\label{isogeny1}
		\lambda_R(q_{E_1}) = \lambda_R\left(\zeta q_{E_2}^{\frac{m}{n}}\right) = \lambda_R(\zeta) + \lambda_R (q_{E_2}^{\frac{m}{n}}) = \frac{m}{n}\lambda_R (q_{E_2}),
	\end{equation}
	\begin{equation}\label{isogeny2}
		v_R(q_{E_1}) = v_R(\zeta q_{E_2}^{\frac{m}{n}}) = v_R(\zeta) + v_R(q_{E_2}^{\frac{m}{n}}) = \frac{m}{n}v_R(q_{E_2}).
	\end{equation}
By the definition of modular symbols (see Equation \eqref{modsymb}) we have the following relation between the modular symbol of $E_1$ and $E_2$ at a point $\frac{a}{b}\in \bP^1(\Q)$ 
	\begin{equation}\label{isogeny3}
		\Omega_{E_1}\modsymb{a}{b}_{E_1} = \Omega_{E_2} \modsymb{a}{b}_{E_2}.
	\end{equation}
So $\modsymb{a}{b}_{E_1} = \frac{\Omega_{E_2}}{\Omega_{E_1}}\modsymb{a}{b}_{E_2}$. Moreover, note that by \cite[Lemma 2.3]{DD15} we have $\frac{\Omega_{E_2}}{\Omega_{E_1}}\in \Q^{\times}$ and $\mathrm{ord}_{\ell}\left(\frac{\Omega_{E_2}}{\Omega_{E_1}}\right) = 0$.

Finally, using Equations \eqref{isogeny1}, \eqref{isogeny2}, and \eqref{isogeny3} in Equation \eqref{equation_twisted}, we conclude the lemma.
\end{proof}

From the previous lemma, we see that Assumption \ref{Q_isogeny} is redundant if we assume that $\ell$ is coprime to the modular degree of $E$. In fact, denote by $\tilde{E}$ the optimal elliptic curve in the $\bQ$-isogeny class of $E$. Recall that every parametrization $\phi:X_0(N) \rightarrow E$ factors through $\tilde{E}$. So, the assumption that $\ell$ is coprime to the modular degree of $E$ is equivalent to the assumption that $\ell$ is coprime to the modular degree of $\tilde{E}$ and the isogeny degree of $\tilde{E} \rightarrow E$. Finally, using Lemma \ref{lemma_assumption} allows us to see that Theorem \ref{theorem_1} is invariant under isogenies of degree coprime to $\ell$. Therefore, if $\ell$ is coprime to the modular degree of $E$, then we can assume, without loss of generality, that $E$ is isomorphic to $\tilde{E}$.

\begin{rmk} Even if $\mathrm{cond}(E)= p$, it is not clear that Conjecture \ref{conjetura twisteada} is a consequence of the work by de-Shalit \cite{deShalit}. In fact, in \cite{deShalit} it is assumed that $\mathrm{cond}(E)= p$, and if the character $\chi$ is of conductor $c> 1$, then the twisted elliptic curve gives rise to a weight $2$ modular form of level $pc$. 
\end{rmk}

\section{Periods and the Exceptional zero phenomena}\label{Periods and the Exceptional zero phenomena}

In this section we will relate the conjectures stated in the previous section and the study of the periods attached to the elliptic curves. 

Let $E$ be an elliptic curve, $\Psi_{\nu}$ an optimal embedding of conductor $c$ with $(\nu, c) = 1$ and $\ell > 3$ a prime coprime to the degree of the strong Weil parametrization of $E$.

\begin{thm}\label{teorema_fuerte}
\begin{enumerate}[label=(\alph*)]
\item\label{item 1} Assume that $c\equiv 1 \pmod{p}$ and $\ell \nmid \phi(c)$.  If conjecture \ref{conjetura twisteada} is true, then 
\begin{equation}\label{e:refined L-invariants}
v_R(q_E) \cdot \lambda_{R}(I_{\Psi})= \lambda_{R}(q_E)\cdot v_R(I_{\Psi}).
\end{equation}
\item\label{item 2} If $\mathrm{cond}(E)= p$, the prime $\ell$ is coprime to the degree of the modular parametrization of $E$, and the conductor of $\Psi$ is $1$, then Equation \eqref{e:refined L-invariants} holds unconditionally and $v_R(q_E)\in R^{\times}$
\end{enumerate}
\end{thm}

\begin{proof} Using Propositions \ref{ordPsi} and $\ref{logPsi-nueva}$, we see that to prove Equation (\ref{e:refined L-invariants}) it is enough to prove
\begin{equation}\label{qeproof}
     v_R(q_E)\cdot \left(\sum_{a\in J_1} \lambda_R(a) \modsymb{a}{pc}\right) = \lambda_R(q_E) \cdot \left(\sum_{a\in J} \modsymb{a}{c}\right).
\end{equation}
To prove \ref{item 1} we assume that Conjecture \ref{conjetura twisteada} is true; that is, for every character $\chi: (\Z/c\Z)^{\times}\rightarrow \C_p^{\times}$ such that $\chi(p) = 1$ we have that the following equality
\begin{equation}\label{qeproof2}
	v_R(q_E) \left(\sum_{a\in (\mathbb{Z}/pc\mathbb{Z})^{\times}} \chi(a) \otimes \lambda_R(a)\modsymb{a}{pc}\right) = \lambda_{R}(q_E)\sum_{a\in (\mathbb{Z}/c\mathbb{Z})^{\times}} \left(\chi(a) \otimes \modsymb{a}{c}\right)  \text{ in } \Z[\mathrm{Im}(\chi)] \otimes R.
\end{equation} 
Let $G = (\Z/c\Z)^{\times}/\left<p\right>$ and by $G^{\vee} = \mathrm{Hom}(G, \mathbb{Q}/\mathbb{Z})$ its Pontryagin dual. Then for each $a\in \Z$ coprime to $c$ we have
\[
	\sum_{\chi\in G^{\vee}} \chi\left(\frac{a}{\nu}\right) = \begin{cases}
	|G| & \text{if} \ \frac{a}{\nu}\equiv p^j\ (\mathrm{mod}\ c) \ \text{for some} \ j \in \Z,\\
	0 & \text{otherwise}.
	\end{cases}
\]
Multiplying Equation (\ref{qeproof2}) by $\chi(\nu^{-1})$ and taking the sum over all $G^{\vee}$ we obtain the following equality
\[
	\sum_{\chi \in G^{\vee}} v_R(q_E) \left(\sum_{a\in (\mathbb{Z}/pc\mathbb{Z})^\times} \chi\left(\frac{a}{\nu}\right) \otimes \lambda_R(a)\modsymb{a}{c}\right) = \sum_{\chi \in G^{\vee}} \lambda_{R}(q_E)\sum_{a\in (\mathbb{Z}/c\mathbb{Z})^\times} \left(\chi\left(\frac{a}{\nu}\right) \otimes \modsymb{a}{c}\right)  \text{ in } \Z[\mathrm{Im}(\chi)] \otimes R.
\]
From this and the above observation we deduce the following equality
\[
	|G|\cdot v_R(q_E) \cdot \left(\sum_{a\in J_1} \lambda_R(a) \modsymb{a}{pc}\right) = |G|\cdot \lambda_R(q_E)\cdot \left(\sum_{a\in J} \modsymb{a}{c}\right) \text{ in } R.
\]
Finally, the hypothesis $(\ell, \phi(c)) = 1$ implies $(\ell, |G|) = 1$, and so we can divide the previous equation by $|G|$, and  part \ref{item 1} of the theorem follows.

Part \ref{item 2} of the theorem is a direct consequence of \ref{item 1} and \cite[Thm. 0.3 and \S6.3]{deShalit}. 
\end{proof}

\begin{rmk} 
As mentioned in \cite{deShalit}, the condition that $\ell$ is coprime to the degree of the strong Weil parametrization of $E$ is a consequence of working at the level of the Jacobian of $X_0(p)$. Though, as pointed out by de Shalit in the same article, if $\ell \geq 5$ then $p-1$ and the degree of the strong Weil parametrization rarely have any common factors. 
\end{rmk}

\begin{rmk}
As mentioned in the introduction, our main theorem only works when $c = 1$. This is a consequence of refining Darmon's proof of \cite[Theorem 1]{Dar01}, as then we need to use de-Shalit results which are in turn refined versions of the Greenberg-Stevens main theorem in \cite{greenberg-stevens}. De-Shalit results are only proven in very restrictive cases. So, we should be able to prove a our theorem for any conductors of the elliptic curve and the optimal embedding, if Conjecture \ref{conjetura twisteada} is true in full generality.

A possible way of proving Conjecture \ref{conjetura twisteada} could be done, for example, by generalizing \cite[Theorem 0.3]{deShalit} for elliptic curves of arbitrary conductor. One of the main step in proving such result was a ``refined'' Hida theory. For the proof, de Shalit used results in \cite{deShalit95} which only cover the case when the elliptic curve has a prime conductor. But in \cite{MR1881569} a more general Hida theory is developed. Therefore, it is expected that such a result could be used following \cite{deShalit} to generalize the theorems needed.
\end{rmk}

\section{The $p$-adic valuation of the $j$-invariant} \label{s: j-invariant}
Let $E/\bQ$ be a rational elliptic curve with split multiplicative reduction at a prime $p$ and $\ell|p-1$ be a prime number. We use the same notation as the previous sections, in particular, we denote by $q_E$ the Tate $p$-adic period. In this section we will not assume Assumption \ref{Q_isogeny}.

In \cite{mazur-tate-teitelbaum} the authors defined the $\cL$-invariant as the ratio $\frac{\log_p(q_E)}{\mathrm{ord}_p(q_E)}$. We are interested in the ``refined $\cL$-invariant'' (see \cite{lecouturier19}, \cite{MR1881569}, and \cite{deShalit}) defined as the ratio $\frac{\lambda_R(q_E)}{v_R(q_E)}\in R$. For the refined $\cL$-invariant to be well-defined, it is necessary that $(\mathrm{ord}_p(q_E), \ell) = 1$. 
\begin{prop}[{de Shalit \cite[\S 6.3]{deShalit}}]\label{result_deShalit}
	If $E$ has conductor $p$, $\ell > 5$ and is coprime to the modular degree of $E$, then $(\mathrm{ord}_p(q_E), \ell) = 1$.
\end{prop}
Thus, under the assumptions of Proposition \ref{result_deShalit}, we can conclude that the refined $\cL$-invariant is well-defined with the hypothesis of Proposition \ref{result_deShalit}. We used SageMath and the Cremona Database to study this fact numerically in the general setting; that is, we checked the following statement.
\begin{expectation}\label{expectation} If $E$ has split multiplicative reduction at $p$, $\ell > 5$ is a prime number such that $\ell \mid p-1$ and is coprime to the modular degree of $E$, then $(\mathrm{ord}_p(q_E), \ell) = 1$.
\end{expectation}

 When $E$ has split multiplicative reduction at a prime $p$ it is known that $-\mathrm{ord}_p(j_E) = \mathrm{ord}_p(q_E)$. Thus we could consider the following, more general statement: \emph{if $E/\bQ$ is a rational elliptic curve, $\ell\mid p-1$ be a prime number with $\ell > 5$ and $\ell$ is coprime to the modular degree, then $(\mathrm{ord}_p(j_E), \ell) = 1$.}
Note that in this statement we could allow $E$ to have any type of reduction at $p$. We used SageMath to numerically check this fact for all the elliptic curves in the Cremona database. We did not find any counter-example. However, we found that the assumption that $\ell > 5$ and the assumption that $\ell$ are coprime to the modular degree do appear to be necessary. We summarize our findings:

\begin{itemize}
	\item The condition that $\ell \ne 5$ is necessary. The elliptic curve $11.a2$ has modular degree $1$ and split multiplicative reduction at $11$. So, if $\ell = 5$ then $\ell$ is coprime to the modular degree and $(\mathrm{ord}_p(q_E), \ell) = 5$. However, this is the only counter-example that we found  when $\ell= 5$.
	\item The condition that $\ell$ is coprime to the modular degree is necessary. Because, if we remove this condition we find multiple counter-examples of Expectation \ref{expectation} in the Cremona database. Some counter-examples are the following\footnote{We will use the L-functions and modular forms database (LMFDB) labelling convention}:
	\begin{itemize}
		\item Consider the elliptic curve $1102.d1$, which has split multiplicative reduction at $29$. Consider $\ell = 7 \mid 29 - 1$. We have that $\mathrm{ord}_{29}(j_E) = -7$ which is not coprime to $\ell$. Remark that in this case the modular degree is $13440= 2^7\cdot 3\cdot 5\cdot 7$.
	 	\item Consider the elliptic curve $111826.a1$, which has split multiplicative reduction at $23$. Consider $\ell = 11 \mid 23 - 1$. We have that $\mathrm{ord}_{23}(j_E) = -11$ which is not coprime to $\ell$. Remark that in this case the modular degree is $25625600= 2^{10}\cdot 5^2 \cdot 7 \cdot 11\cdot 13$.
		\item Consider the elliptic curve $137270.o1$, which has split multiplicative reduction at $53$. Consider $\ell = 13 \mid 53 - 1$. We have that $\mathrm{ord}_{53}(j_E) = -13$ which is not coprime to $\ell$. Remark that in this case the modular degree is $259144704= 2^{10}\cdot 3^{3}\cdot 7\cdot 13\cdot 103 $.
	\end{itemize}
\end{itemize}
The code used to check numerically Expectation \ref{expectation} above and an explanation of it can be found in Github using the following code: \url{https://github.com/JpLlerena/On_Periods_Of_Elliptic_Curves}
\bibliographystyle{alpha}
\bibliography{biblio}

\end{document}